\documentclass[12pt]{amsart}
\usepackage{hyperref}
\usepackage{amsfonts}
\usepackage{amsmath}
\usepackage{amssymb}
\usepackage{amscd}
\usepackage{graphics}
\usepackage{latexsym}
\usepackage{color}
\usepackage{epsfig}
\usepackage{tikz-cd}
\usepackage{amsopn}

\newtheorem{theorem}{Theorem}[section]
\newtheorem{proposition}[theorem]{Proposition}
\newtheorem{lemma}[theorem]{Lemma}
\newtheorem{corollary}[theorem]{Corollary}

\newtheorem{problem}[theorem]{Problem}

\newtheorem*{theorem*}{Theorem}
\newtheorem*{lemma*}{Lemma}
\newtheorem*{proposition*}{Proposition}
\newtheorem*{conjecture*}{Conjecture}
\newtheorem*{corollary*}{Corollary}
\newtheorem*{problem*}{Problem}

\theoremstyle{remark}
\newtheorem{remark}[theorem]{Remark}
\newtheorem{notation}[theorem]{Notation}

\newtheorem{definition}[theorem]{Definition}

\newcommand{\PP}{\mathbb{P}}

\DeclareMathOperator{\Gr}{Gr}

\DeclareMathOperator{\Span}{Span}

\DeclareMathOperator{\PGL}{\mathbb{P}GL}

\begin{document}

\title{$\PGL$ orbits on products of flag varieties}
\author[I. Coskun]{Izzet Coskun}
\address{Department of Mathematics, Stat. and CS \\University of Illinois at Chicago, Chicago, IL 60607}
\email{icoskun@uic.edu}
\author[A. G\"und\"uz]{Abuzer G\"und\"uz}
\address{Sakarya University, Mathematics Department, Esentepe Campus, Serdivan, Sakarya 54050 T\"{u}rkiye}
\email{abuzergunduz@sakarya.edu.tr}

\thanks{During the preparation of this article the first author was partially supported by the NSF grant DMS-2200684, the second author was supported by T\"{U}BITAK 2219 grant}
\begin{abstract}
In this paper, we study the existence of a dense orbit for the diagonal $\PGL(n)$ action on self-products of partial flag varieties. We determine when there exists a dense orbit for flag varieties of the form $F(k_1, \dots, k_r; n)^m$   when $k_i = i$ or when $m=3$ and $k_i = \ell + i$ for some $\ell \geq 0$. We also show that certain infinite families of products do not have a  dense orbit.
\end{abstract}
\maketitle

\section{Introduction}
The problem of finding convenient coordinates for a collection of linear spaces is as old as coordinate geometry. The existence of easy to work with coordinates is often at the heart of many classical theorems. The prototypical result states that given two ordered sets $\{p_1, \dots, p_{n+2}\}$ and $\{q_1, \dots, q_{n+2}\}$ of $n+2$ points in $\PP^n$ in general linear position, there exists a unique  $M \in \PGL(n+1)$ such that $M(p_i) = q_i$ for $1 \leq i \leq n+2$ (see \cite[Section 1.6]{harris}). Consequently, the diagonal action of $\PGL(n+1)$ on $\prod_{i=1}^m\PP^n$ has a dense orbit if and only if $m \leq n+2$. In this paper, we consider generalizations of this fact to self-products of partial flag varieties.

Let $0 < k_1 < \cdots < k_r <n$ be positive integers.  For notational convenience, set $k_0=0$ and $k_{r+1}=n$.  Let $V$ be an $n$-dimensional vector space. Let $F(k_1, \dots, k_r; n)$ denote the partial flag variety parameterizing the set of partial flags $$W_{k_1} \subset  \cdots \subset W_{k_r} \subset V,$$ where $\dim (W_{k_i}) = k_i$. The group $\PGL(n)$ acts transitively on $F(k_1, \dots, k_r; n)$. This action induces a diagonal action on  products of the form $\prod_{i=1}^m F(k_{i, 1} , \dots, k_{i, r_i}; n)$. We say that the action is {\em dense} if there exists a dense orbit of the action. Otherwise, we say the action is {\em sparse}. In this paper, we address special cases of the following problem.

\begin{problem}
Classify products $F(k_1, \dots, k_r; n)^m$ for which $\PGL(n)$ acts with dense orbit.
\end{problem}

Our results on density can be  summarized in the following theorems. We first fully resolve the case when $k_i = i$ for $1 \leq i \leq r$.

\begin{theorem*} [Theorem \ref{thm-1r}]
The $\PGL(n)$ action on $F(1, 2, \dots, r-1, r; n)^m$ is dense if and only if one of the following holds:
\begin{enumerate}
\item $m \leq 2$; or 
\item $r=1$ and $m \leq n+1$; or
\item $m=3$, $r>1$, and $n \geq 3r-2$; or
\item $m\geq 4$, $r>1$, and $n \geq 3r-1$.
\end{enumerate}
\end{theorem*}

As a natural generalization, we consider the case when $k_i = \ell + i$ for $1 \leq i \leq r$. We obtain a complete classification when $m=3$ or when $m$ is arbitrary and $n \geq (m-1)(\ell + r)$. 

\begin{theorem*}[Theorem \ref{thm-1apart}]
The $\PGL(n)$ action on the product $F(\ell + 1, \dots, \ell +r; n)^3$ is sparse if $2\ell + r < n < 2\ell + 2r$. Assume that $n \geq 2 \ell + 2r$. Then the $\PGL(n)$ action on $F(\ell + 1, \dots, \ell +r; n)^3$  is dense if and only if one of the following holds:
\begin{enumerate}
\item $r=1$; or
\item $n \geq 3\ell + 3r$ or $n = 3\ell + 3r -2$; or 
\item $2\ell + 2r \leq n \leq 3\ell + 2r$ and either $r=2$ or $2n \geq 3\ell + 3r -2$.
\end{enumerate}
\end{theorem*}

\begin{remark}
The cases when $\ell + r < n \leq 2\ell + r$ are determined by the duality exchanging $k_i$ with $n-k_i$ and are listed in detail in Theorem \ref{thm-1apart}.
\end{remark}

\begin{proposition*} [Proposition \ref{prop-1apart}]
 Assume that $m \geq 4$ and $n \geq (m-1) (\ell + r)$.  Set $$j = n - (m-1)(\ell+r).$$ Then the $\PGL(n)$ action on $F(\ell+1, \dots, \ell+r; n)^m$ is dense if and only if one of the following holds:
 \begin{enumerate}
  \item $n \geq m(\ell + r)$; or
 \item $r=1$ and $m(\ell+1)(n-\ell-1) < n^2$; or
  \item $0 \leq j < \ell$ and either $r=2$ and $m \leq \ell-j+r+1$ or $r > 2$ and $\ell-j+r \geq mr-m-1$.
 \end{enumerate}
\end{proposition*}

In the other extreme, we study the cases where the differences $k_i - k_{i-1}$ are large. Our results in this direction include the following.

\begin{theorem*} [Proposition \ref{prop-m1} and Corollary \ref{cor-farapart}] The $\PGL(n)$ action on  $F(k_1, \dots, k_r; (m-1)k_r)^m$ is dense if and only if the $\PGL(k_r)$ action on $F(k_1, \dots, k_{r-1}; k_r)^m$ is dense. In particular, if $k_i \geq (m-1)k_{i-1}$ for $2 \leq i \leq r+1$, then the $\PGL(n)$ action on $F(k_1, \dots, k_r; n)^m$ is dense.
\end{theorem*}

We also exhibit  infinite families of products  with sparse $\PGL(n)$ action.

\begin{theorem*} [Proposition \ref{prop-sparse} and Proposition \ref{prop-obstruction}]

\begin{enumerate}
\item The $\PGL(n)$ action on $F(k_1, \dots, k_r; n)^m$ is sparse if there exists two indices $1 \leq i<j \leq r$ such that $k_j = n - t k_i$ for some $t \geq 1$ and $m \geq \max(4, t+2)$.
\item The $\PGL(n)$ action on $F(k_1, \dots, k_r; n)^3$ is sparse if there exists three indices $1 \leq h< i< j \leq r$ such that $2k_h + k_i + k_j= 2n$.
\end{enumerate}
\end{theorem*}

\noindent For example, we conclude that the $\PGL(n)$ action on $F(k, n-2k; n)^m$ is dense if and only if $m \leq 3$. Similarly, given $k_1<k_2$ such that $k_1 + k_2$ is even, the $\PGL(n)$ action on $F(k_1, k_2, n- \frac{k_1+k_2}{2}; n)^3$ is sparse.

\subsection*{History of the problem and applications} The problem of classifying algebraic group actions with dense orbit has a long history. Special cases have played an important role in classical geometry.   The systematic study of products of flag varieties with dense $\PGL$ orbits was initiated by Popov \cite{Popov, Popov2} based on a question of M. Burger and further studied in \cite{CEY, CHZ, Devyatov}. We refer the reader to \cite{Smirnov} for a recent survey. 

More generally, Popov studied the problem for connected algebraic groups  and classified the simple linear algebraic groups $G$ such that the action of $G$ on $(G/P)^n$ is dense when $P$ is a maximal parabolic subgroup \cite{Popov, Popov2}. Devyatov has extended the  classification to non-maximal parabolic subgroups $P$, except in type $A$ \cite{Devyatov}. When $G/P$ is a type A partial flag variety, even the classification for $(G/P)^3$ is unknown. The closely related but easier problem of  characterizing products of  partial flag varieties with finitely many  $\PGL(n)$ orbits was solved by Magyar, Weyman and Zelevinsky \cite{MWZ}.

 The density of the $\PGL(n)$ action on a product of flag varieties  has many applications; we refer the reader to \cite{CHZ, Popov2} for details.  One of the main motivations for Popov's work was to obtain bounds on certain Littlewood-Richardson coefficients and prove that they are multiplicity-free. More explicitly, let $\lambda_1, \dots, \lambda_d$  be nonzero dominant characters of the maximal torus $T$ in the semi-simple group $G$. Then $(\lambda_1, \dots, \lambda_d)$  is called {\em primitive} if for every nonnegative $d$-tuple of integers $(n_1, \dots, n_d)$, the Littlewood-Richardson coefficient $c^0_{n_1\lambda_1,\dots ,n_d \lambda_d} \leq 1$. Popov in \cite[Theorem 1]{Popov2} proves that if $G$ has an open orbit on $G/P_{\lambda_1} \times \cdots \times G/P_{\lambda_d}$, then $(\lambda_1, \dots, \lambda_d)$ is primitive.  Similarly, the density has  geometric applications to the study of vector bundles, enumerative geometry and genus zero Gromov-Witten invariants (see \cite{CHZ} for more details).

\subsection*{Organization of the paper} In \S \ref{sec-prelim}, after discussing basic facts concerning the geometry of partial flag varieties and group actions, we recall key results from \cite{CHZ} and \cite{CEY} concerning the $\PGL(n)$ action on products of flag varieties. In \S \ref{sec-density}, we characterize the density of the action on $F(k_1, \dots, k_r; n)^m$ when $k_i = i$ for $1 \leq i \leq r$ or when $k_i \geq (m-1) k_{i-1}$ for $2 \leq i \leq r$. Finally, in \S \ref{sec-sparse}, we give further examples of sparse actions on products of partial flag varieties.

\subsection*{Conflict of interest statement} The authors have no conflict of interest to declare that are relevant to this article.
\subsection*{Data availability statement} There is no data associated to this article.

\subsection*{Acknowledgements} We would like to thank Dave Anderson, Demir Eken, James Freitag, Majid Hadian, J\'{a}nos Koll\'{a}r, Howard Nuer, Sybille Rosset, Geoffrey Smith, Chris Yun and Dmitry Zakharov for helpful discussions regarding actions of $\PGL(n)$ on products of varieties.

\section{Preliminaries}\label{sec-prelim}

In this section, we collect basic facts concerning the diagonal action of $\PGL(n)$ on products of flag varieties. We refer the reader to \cite{CHZ, CEY} for proofs and additional  details.

\subsection{Partial flag varieties} Let $V$ be an $n$-dimensional vector space.  For an integer $0 < k < n$, let $\Gr(k,n)$ denote the Grassmannian parameterizing $k$-dimensional subspaces of $V$. It is a projective variety of dimension $k(n-k)$. 

More generally, let $0 < k_1 < \cdots < k_r < n$ be a sequence of $r$  integers. For notational convenience, set $k_0=0$ and $k_{r+1}=n$.  Let $F(k_1, \dots, k_r; n)$ denote the  flag variety parameterizing partial flags $$W_{k_1} \subset  \cdots \subset W_{k_r} \subset V,$$ where $\dim (W_{k_i}) = k_i$. The flag variety   $F(k_1, \dots, k_r; n)$  is a projective variety of dimension $$\dim(F(k_1, \dots, k_r; n)) = \sum_{i=1}^r k_i(k_{i+1}- k_i).$$ The group $\PGL(n)$ acts transitively on $F(k_1, \dots, k_r; n)$. In this paper, we will be concerned with the induced diagonal action on $\prod_{i=1}^m F(k_{i, 1} , \dots, k_{i, r_i}; n)$.

 We say that the action is {\em dense} if there exists a dense orbit of the action. Otherwise, we say that the action is {\em sparse}.  If the action of $\PGL(n)$ on $\prod_{i=1}^m F(k_{i, 1} , \dots, k_{i, r_i}; n)$ is dense, then $$\dim (\PGL(n))= n^2 -1 \geq \sum_{i=1}^m \sum_{j=1}^{r_i} k_{i,j}(k_{i,j+1}- k_{i,j})= \dim \left(\prod_{i=1}^m F(k_{i, 1} , \dots, k_{i, r_i}; n)\right).$$ If this inequality is violated, we say that the action is {\em trivially sparse}. It is easy to give examples of actions that are sparse but not trivially sparse (see \cite{CHZ, CEY}).

We will often prove the equivalence of the density of two different products of flag varieties. The following notation will be useful.  
\begin{notation}
Given a variety $X$ with a $\PGL(n)$ action and a variety $Y$ with a $\PGL(m)$ action, we write $X \Leftrightarrow Y$ to mean that the density of the action on $X$ is equivalent to the density of the action on $Y$. 
\end{notation}

\subsection{Dense orbits and the dimension of stabilizers} Our analysis depends on an elementary observation. If $G$ is an algebraic group acting on an irreducible projective variety $X$,  then the orbit $Gx$ of a closed point $x$ under $G$ is open in its Zariski closure $\overline{Gx}$ by \cite[I.1.8]{Borel}. Let $H$ denote the stabilizer of $x$. Since $X$ is irreducible, the orbit of $x$ is dense if and only if $$\dim (Gx) = \dim(\overline{Gx}) = \dim(X).$$ Since $Gx$ is isomorphic to $G/H$, we conclude that the orbit of $x$ is dense if and only if $$\dim(H)= \dim(G) - \dim(X).$$ Observe that since $\dim(Gx) \leq \dim(X)$, the inequality $\dim(H) \geq \dim(G) - \dim(X)$ always holds. To verify the  density of the  orbit of  $x$, one has to check the inequality $\dim(H) \leq \dim(G) - \dim(X)$. We say  $\dim(G) - \dim(X)$ is {\em the expected dimension} of the stabilizer. 

\begin{lemma}\label{lem-basiclem}\cite[Lemma 2.2]{CHZ}
Let $X$ be an irreducible projective variety with a $\PGL(n)$ action. Let  $x \in X$ be a closed point and let $H$ denote the stabilizer of $x$. Then the orbit of $x$ is dense in $X$ if and only if $$\dim(H) = n^2-1 - \dim(X).$$ In particular, the diagonal action of $\PGL(n)$ on $\prod_{i=1}^m F(k_1, \dots, k_r; n)$ is dense if and only if there exists a closed point in $\prod_{i=1}^m F(k_1, \dots, k_r; n)$ whose stabilizer has dimension (at most) $$n^2 -1 - m \sum_{i=1}^r k_i (k_{i+1} - k_i).$$
\end{lemma}

\subsection{Classification of dense actions in special cases}
The dense diagonal actions on products of at most four Grassmannians have been classified (see \cite{CHZ, SW}).

\begin{theorem}\cite[Theorem 5.1]{CHZ}\label{thm-Grassmannian4}
Let $m \leq 4$ and let  $X=\prod_{i=1}^m \Gr(k_i, n)$ be a product of $m$-Grassmannians. Then the $\PGL (n)$  action on $X$ is sparse if and only if $m=4$ and $\sum_{i=1}^4 k_i = 2n$.
\end{theorem} 

Similarly, the dense diagonal actions on self-products of Grassmannians have been classified by Popov \cite[Theorem 3]{Popov} and generalized to  $\prod_{i=1}^m \Gr(k_i, n)$ with $|k_i - k_j|<3$ for all $i,j$ \cite[Theorem 7.2]{CHZ}. We state a special case which will be sufficient for our purposes.

\begin{theorem}\cite[Theorem 7.2]{CHZ}\label{thm-popov}
Let $k<n$ be positive integers and let $a, b$ be nonnegative integers.  The $\PGL(n)$ action on  $\Gr(k-1, n)^a \times \Gr(k,n)^b$ is dense if and only if it is not trivially sparse and  $\Gr(k-1, n)^a \times \Gr(k,n)^b \not= \Gr(k-1, 2k-1)^2 \times \Gr(k,2k-1)^2$.
\end{theorem}

In the case of products of partial flag varieties much less is known. A complete classification is available only for triple self-products of two-step partial flag varieties.

\begin{theorem}\cite[Theorem 5.10]{CEY}\label{thm-three2step}
The $\PGL(n)$ action on  $F(k_1, k_2; n)^3$ is dense if and only if $k_1+ k_2 \not= n$. 
\end{theorem}

The following immediate corollary of Theorem \ref{thm-three2step} will be useful.

\begin{corollary}\label{cor-three2step}
Let $m \geq 3$. The $\PGL(n)$ action on  $F(k_1, \dots, k_r; n)^m$ is sparse if there exists indices $i \not= j$ such that $k_i + k_j = n$.
\end{corollary}

\subsection{Reduction lemmas}
In this subsection, we collect several easy lemmas. 

Consider the action of $\PGL(n)$ on a product $\prod_{i=1}^m F(k_{i,1}, \dots, k_{i, r_i}; n)$.
If $n\geq \sum_{i=1}^m k_{i, r_i}$, then we can choose a basis for the underlying vector space $V$ so that the vector space $W_{i, k_{i,j}}$ is the span of basis elements $e_{\ell}$ with $\sum_{t=1}^{i-1} k_{t,j} < \ell \leq \sum_{t=1}^{i} k_{t,j}$. The stabilizer matrices are then of block diagonal form and the dimension of the stabilizer is the expected dimension. 

\begin{lemma}\cite[Lemma 5.5]{CEY}\label{lem-easydense}
The product $\prod_{i=1}^m F(k_{i,1}, \dots, k_{i, r_i}; n)$ is dense if $\sum_{i=1}^m k_{i, r_i} \leq n$.  
\end{lemma}

There is a natural duality between $F(k_1, \dots, k_r; n)$ and $F(n-k_r, \dots, n-k_1; n)$ sending a vector space $W \subset V$ to the annihilator of $W$ in the dual vector space $V^*$. This action respects the natural $\PGL(n)$ actions and yields the following lemma. 

\begin{lemma}\cite[Proposition 2.7]{CEY}\label{lem-duality}
We have  $$\prod_{i=1}^m F(k_{i,1}, \dots, k_{i, r_i}; n) \Leftrightarrow \prod_{i=1}^m F(n-k_{i,r_i}, \dots, n-k_{i, 1}; n).$$ 
\end{lemma}

\begin{notation}
Let $0< k_1 < \cdots < k_r<n$ be an increasing  sequence of positive integers less than $n$. We will denote the sequence by $k_{\bullet}$.  Let $n' \leq n$ and set $d= n- n'$. Let $u$ be the index such that $k_{u-1} \leq d < k_u$. Then for $1 \leq i \leq r-u+1$, set $k_i' = k_{u+i -1} - d$.   Given a sequence $k_{\bullet}$ and an integer $n' < n$, we will call the sequence $k_{\bullet}'$  {\em the sequence derived from $k_{\bullet}$ with respect to $n$ and $n'$}.   Notice that this only depends on $n-n'$, so if we do not wish to emphasize $n$ and $n'$, we will sometimes say {\em the sequence derived from $k_{\bullet}$ with respect to $d$}. Given a vector space $W$ of dimension $n'$ and a general partial flag $F_{\bullet}$ in an $n$-dimensional vector space with dimensions $k_{\bullet}$,  the sequence $k_{\bullet}'$ denotes the dimension vector of the partial flag in $W$ obtained by $F_{\bullet} \cap W$.
\end{notation}

The following useful lemma will allow us to reduce the dimension. 

\begin{lemma}\cite[Lemma 5.6]{CEY}\label{lem-reduce}
Suppose $k_{i, \bullet}$ are $m$ sequences of increasing positive integers less than $n$. Assume that $n'= \sum_{i=1}^{m-1} k_{i, r_i} \leq n < \sum_{i=1}^m k_{i, r_i}$. Let $k_{m, \bullet}'$ be the sequence derived from $k_{m,\bullet}$ with respect to $n$ and $n'$. Then 
$$\prod_{i=1}^m F(k_{i, \bullet}; n) \Leftrightarrow \prod_{i=1}^{m-1} F(k_{i, \bullet}; n') \times F(k_{m, \bullet}'; n').$$ 
\end{lemma}

\section{Density results}\label{sec-density}
In this section, we prove our main results on density. We first analyze the case where $k_i=i$. We then consider the  other extreme when the difference between $k_i$ and $k_{i+1}$ is large.

\subsection{The case $k_i = i$} We first study the case when $k_i = i$ for $1 \leq i \leq r$. For simplicity, we denote this partial flag variety by $F(1,2, \dots, r-1, r; n)$.

\begin{theorem}\label{thm-1r}
The $\PGL(n)$ action on  $F(1, 2, \dots, r-1, r; n)^m$ is dense if and only if one of the following holds:
\begin{enumerate}
\item \label{i} $m\leq 2$;
\item  \label{ii} $r=1$, and $n \geq m-1$;
\item \label{iii} $m=3$, $r>1$, and $n \geq 3r-2$;
\item \label{iv} $m>3$, $r>1$, and $n \geq mr-1$.
\end{enumerate}
\end{theorem} 

\begin{proof}
\noindent{\bf Step 1: $m\leq 2$.} When $m\leq 2$, the action of $\PGL(n)$ on $F(1, 2, \dots, r-1, r; n)^m$ has only finitely many orbits \cite[Theorem 2.2]{MWZ}, hence the maximal dimensional orbit is dense. This is Case \eqref{i} of the theorem. 
\smallskip

\noindent{\bf Step 2: $r=1$.} When $r=1$, the corresponding partial flag variety is $\PP^{n-1}$. The action of $\PGL(n)$ on $m$ points in $\PP^{n-1}$ has a dense orbit if and only if $m \leq n+1$ (see \cite[Section 1.6]{harris}). This is Case \eqref{ii} of the theorem. 

We may therefore assume that $m \geq 3$ and $r > 1$ for the rest of the proof.
\smallskip

\noindent{\bf Step 3:} We now show that if $n \leq (m-1)r$, then $F(1, 2, \dots, r-1, r; n)^m$ is trivially sparse except when $m=3$ and $r=2$. We have
that
\begin{equation} \label{eq-rdim}
\varphi(n,m,r) := \dim(\PGL(n)) - \dim (F(1, 2, \dots, r-1, r; n)^m) = n^2 -1 - m \frac{(r-1)r}{2} - m r(n-r).
\end{equation}
For fixed $m$ and $r$,  the function $\varphi(n,m,r)$ is a quadratic function of $n$ with leading term $n^2$. Hence, if for fixed $m$ and $r$, the function $\varphi$ takes negative values at $n_1$ and $n_2$, it must take negative values for $n_1 \leq n \leq n_2$ as well.  When $n=r+1$, we have  $$\varphi(r+1, m , r) = \left(1 - \frac{m}{2}\right) r^2 + \left(2-\frac{m}{2}\right)r.$$ This is negative  since $m\geq 3$ and  $r>1$. When $n= (m-1)r$,  we have $$\varphi((m-1)r, m, r) = (m-1)^2 r^2 - 1 - m \frac{(r-1)r}{2} - m(m-2)r^2= \left(1- \frac{m}{2}\right)r^2 + \frac{mr}{2} -1,$$ which is negative when $m\geq 4$ and $r > 1$ or when $m=3$ and $r>2$. We conclude that if $m \geq 4$ and $r>1$ or when $m=3$ and $r>2$, then  $\dim(F(1, 2, \dots, r-1, r; n)^m) > \dim(\PGL(n))$  for $r+1 \leq n \leq (m-1)r$. Hence, the $\PGL(n)$ action on $F(1, 2, \dots, r-1, r; n)^m$ is trivially sparse if $n \leq (m-1)r$ except when $m=3$ and $r=2$.  By Theorem \ref{thm-three2step}, the $\PGL(4)$ action on $F(1,2;4)^3$ is dense. This case is covered by Case \eqref{iii} of the theorem.  We may therefore assume that $n > (m-1)r$. 
\smallskip

\noindent{\bf Step 4: $n \geq mr$.} If $n\geq mr$, then the $\PGL(n)$ action on $F(1, 2, \dots, r-1, r; n)^m$ is dense by Lemma \ref{lem-easydense}. Hence, we may assume that $(m-1)r < n < mr$. 
\smallskip

\noindent{\bf Step 5: $(m-1)r < n < mr$.}  Write $n= mr-j$ for some $0 < j < r$. By applying Lemma \ref{lem-reduce} repeatedly $m$ times, we have  $$F(1, 2, \dots, r-1, r; n)^m \Leftrightarrow F(1, 2, \dots, j-1, j; (m-1)j)^m.$$ If $m \geq 4$, this is trivially sparse if $j>1$ and dense if $j=1$ by  Steps 2 and  3. This yields Case \eqref{iv} of the theorem. If $m=3$, this is trivially sparse if $j >2$ and dense if $1 \leq j \leq 2$ by Steps 2 and 3. This yields Case \eqref{iii} of the theorem and concludes the proof.
\end{proof}

\subsection{The case of sequences with large gaps} Next, we study the case when the sequence $k_1, \dots, k_r$ has large gaps. Our analysis is based on the following observation.

\begin{proposition}\label{prop-m1}
We have  $$F(k_1, \dots, k_r; (m-1)k_r)^m \Leftrightarrow F(k_1, \dots, k_{r-1}; k_r)^m.$$ 
\end{proposition}

\begin{proof}
We first compute the expected dimension of the generic stabilizer in each case.
\begin{align*}
& \dim (\PGL((m-1)k_r)) - \dim (F(k_1, \dots, k_r; (m-1)k_r)^m) \\
&=(m-1)^2k_r^2-1 - m \left(\sum_{i=1}^{r-1} k_i (k_{i+1} - k_i) + (m-2)k_r^2\right)\\
&=k_r^2-1- m \sum_{i=1}^{r-1} k_i (k_{i+1} - k_i) \\
&=\dim  (\PGL(k_r)) -  \dim (F(k_1, \dots, k_{r-1}; k_r)^m). 
\end{align*}
We conclude that the expected dimensions of the generic stabilizers are equal for both products. In particular, the action on one is trivially sparse if and only if the action on the other is.

We now give a construction of an element in $F(k_1, \dots, k_{r-1}; k_r)^m$ given a general element in $F(k_1, \dots, k_r; (m-1)k_r)^m$.
For $1 \leq i \leq m$, let $W_{i, k_1} \subset \cdots \subset W_{i, k_r}$ be general partial flags in $F(k_1, \dots, k_r; (m-1)k_r)$. Let $U_j$ be the span of $W_{i, k_r}$ for $1 \leq i \leq m-1$ and $i \not= j$.  Note that $$\dim(U_j) = (m-2)k_r.$$ For $1 \leq j \leq m-1$ and $1 \leq \ell \leq r$, set $T_{j, \ell}$ be the span of $W_{j, k_\ell}$ and $U_j$. Observe that $$\dim(T_{j , \ell}) = (m-2)k_r + k_{\ell}.$$  Set $W_{j, k_\ell}' = T_{j, \ell} \cap W_{m, k_r}$, which is a $k_\ell$-dimensional vector space. Finally, set $W_{m, k_i}' = W_{m, k_i}$. Note that $W_{j, k_r}' = W_{m, k_r}$ for every $j$ since $U_j$ and $W_{j, k_r}$ span the  underlying vector space $V$.

It is easy to see that a general element in $F(k_1, \dots, k_{r-1}; k_r)^m$ arises via this construction. Given $W_{m, k_r} \subset V$ and general partial flags $W_{j, k_1}' \subset \cdots \subset W_{j, k_{r-1}}'$, fix $m-1$ general $k_r$-dimensional vector spaces $W_{j, k_r} \subset V$. As in the previous paragraph, let $U_j$ be the span of $W_{i, k_r}$ for $1 \leq i \leq m-1$ and $i \not= j$. Set $W_{j, k_{\ell}}$ to be the intersection of $W_{j, k_r}$ with the span of $U_j$ and $W_{j, k_r}'$. This is clearly the inverse construction.  Hence, if $F(k_1, \dots, k_{r-1}; k_r)^m$ is not dense, then $F(k_1, \dots, k_r; (m-1)k_r)^m$ cannot be dense. 

To prove the converse, we argue as follows. Let $S$ denote the stabilizer of the vector spaces $W_{j, k_i}$ in $V$. Let $S'$ denote the stabilizer of the vector spaces $W_{j, k_i}' $ in $W_{m, k_r}$. Since $W_{j, k_i}'$ are obtained from $W_{j, k_i}$ by taking spans and intersections, any stabilizer of  the vector spaces $W_{j, k_i}$ in $V$ also stabilizes $W_{j, k_i}'$. We thus get a group homomorphism 
$f:S \to S'.$ By \cite[Lemma 4.5]{CHZ}, the kernel of $f$ is trivial. Hence, we obtain the  inequality 
$\dim(S) \leq \dim(S').$
If the $\PGL(k_r)$ action on $F(k_1, \dots, k_{r-1}; k_r)^m$ is dense, then $$\dim(S') = k_r^2-1- m \sum_{i=1}^{r-1} k_i (k_{i+1} - k_i) \leq \dim(S).$$  
We deduce that we must have equality  in the dimensions and  the $\PGL((m-1)k_r)$ action on $F(k_1, \dots, k_r; (m-1)k_r)^m$ is dense by Lemma \ref{lem-basiclem}. 
\end{proof}

\begin{definition}
Given a sequence of positive integers $0 < k_1 < \cdots <  k_r < n$ and an integer $m > 2$, we define integers $j_\ell(m)$ inductively by descending induction starting with $\ell=r$ as follows. We set $$j_r (m):= n - (m-1) k_r.$$  Then for $0 < \ell < r$, we inductively set $$j_{\ell}(m) = k_{\ell+1} - (m-1)  k_{\ell}  + (m-2)  \sum_{i=\ell+1}^r j_i (m).$$ 
If $k_{\ell} -  \sum_{t=\ell}^r j_t(m)>0$, define $s_{\ell}$ by $$s_{\ell} := \min\left\{ 0 < i \leq \ell \ | \ k_i - \sum_{t=\ell}^r j_t(m)>0\right\};$$ otherwise set $s_{\ell} = \ell$.  Let $$\ell_0 := \max\{\ell\ | \ s_{\ell} \geq \ell -1\}.$$ 

\end{definition}

With this notation, we have the following theorem.

\begin{theorem}\label{thm-mkr}
Let $m>2$ be a positive integer. Let $0< k_1 < \cdots < k_r < n$ be a sequence of integers such that  $j_\ell(m) \geq  0$ for  all $\ell \geq \ell_0$. Then the $\PGL(n)$ action on $F(k_1, \dots, k_r; n)^m$ is dense if one of the following holds:
\begin{enumerate}
\item $j_t(m) \geq k_t - \sum_{\ell = t+1}^r j_{\ell}(m)$ for some $t \geq \ell_0$; or
\item $s_{\ell_0} \geq \ell_0$; or
\item $s_{\ell_0} = \ell_0 -1$ and $m \left(k _{\ell_0 -1} - \sum_{\ell= \ell_0}^r j_\ell(m)\right)(k_{\ell_0} - k_{\ell_0 -1}) < \left(k_{\ell_0}- \sum_{\ell=\ell_0}^r \ell_t(m)\right)^2$.
\end{enumerate}
\end{theorem}

\begin{proof}
Since $m$ is fixed throughout the proof, we denote  $j_\ell(m)$ simply by $j_\ell$. We will prove by descending induction on $\ell$ that if $\ell > \ell_0$, then  either the action is dense or $$F(k_1, \dots, k_r; n)^m \Leftrightarrow F\left(k_{s_\ell} - \sum_{u= \ell}^r j_u, \dots, \ k_{\ell -1} - \sum_{u= \ell}^r j_u; \ k_{\ell} - \sum_{u= \ell}^r j_u\right)^m.$$

If $n \geq m k_r$,  then $j_r \geq k_r$ and the $\PGL(n)$ action on $F(k_1, \dots, k_r; n)^m$ is dense by Lemma \ref{lem-easydense}. Otherwise, by assumption $j_r\geq 0$, hence $mk_r > n \geq (m-1) k_r$. By $m$ applications of  Lemma \ref{lem-reduce},we have  $$F(k_1, \dots, k_r; n)^m \Leftrightarrow F(k_{s_r} - j_r, \dots, k_r - j_r; (m-1)(k_r-j_r))^m.$$   If $s_r = r$, then the action is dense by Theorem \ref{thm-popov}. Otherwise,  by Proposition \ref{prop-m1}, $$F(k_{s_r} - j_r, \dots, k_r - j_r; (m-1)(k_r-j_r))^m \Leftrightarrow F(k_{s_r} - j_r, \dots, k_{r-1} - j_r; k_{r} - j_r)^m.$$ This proves the base case of the induction.  If $s_r = r-1$, then the latter is an $m$-fold product of a Grassmannian and by Theorem \ref{thm-popov}, the action is dense if and only if $$m (k_{r-1} - j_r)(k_r-k_{r-1}) < (k_r-j_r)^2.$$ The theorem thus holds for $\ell = r$. 

Suppose that by induction we have $$F(k_1, \dots, k_r; n)^m \Leftrightarrow F\left(k_{s_\ell} - \sum_{u= \ell}^r j_u, \dots, k_{\ell -1} - \sum_{u= \ell}^r j_u; k_{\ell} - \sum_{u= \ell}^r j_u\right)^m.$$ If $j_{\ell -1} \geq k_{\ell-1} - \sum_{u= \ell}^r j_u$, then $$k_{\ell} - \sum_{u= \ell}^r j_u \geq m \left(k_{\ell -1} - \sum_{u= \ell}^r j_u\right)$$ and the action on the latter is dense by Lemma \ref{lem-easydense}. Otherwise, by the assumption that $j_{\ell-1} \geq 0$, we have that  $$k_{\ell} - \sum_{u= \ell}^r j_u \geq (m-1) \left(k_{\ell -1} - \sum_{u= \ell}^r j_u\right).$$
By $m$ applications of  Lemma \ref{lem-reduce},  $$F\left(k_{s_\ell} - \sum_{u= \ell}^r j_u, \dots, k_{\ell -1} - \sum_{u= \ell}^r j_u; k_{\ell} - \sum_{u= \ell}^r j_u\right)^m \Leftrightarrow$$ $$ F\left(k_{s_{\ell-1}} - \sum_{u= \ell-1}^r j_u, \dots, k_{\ell -1} - \sum_{u= \ell-1}^r j_u; (m-1)(k_{\ell -1} - \sum_{u= \ell-1}^r j_u)\right)^m.$$ If $s_{\ell -1} = \ell -1$, then the action is dense by Theorem \ref{thm-popov}. Otherwise, by Proposition \ref{prop-m1}, the density of the action on the latter is equivalent to the density of the action on $$F(k_{s_{\ell-1}} - \sum_{u= \ell-1}^r j_u,\dots; k_{\ell -1} - \sum_{u= \ell-1}^r j_u)^m.$$ This concludes the inductive step. The induction stops when $\ell = \ell_0$. In which case, either the action is dense or we have the action on a product of Grassmannians whose density is determined by Theorem \ref{thm-popov}. This concludes the proof of the theorem.
 \end{proof}
 
 In particular, we obtain the following easier to remember corollary.
 
 \begin{corollary}\label{cor-farapart}
 If $k_i \geq (m-1) k_{i-1}$ for $2 \leq i \leq r+1$, then the $\PGL(n)$ action on $F(k_1, \dots, k_r; n)^m$ is dense.
 \end{corollary}
 
 Applying Theorem \ref{thm-mkr} when $r=2$, we obtain the following.
 
 \begin{corollary}
 Let $m > 3$ and assume that $n \geq (m-1)k_2$. Then the $\PGL(n)$ action on $F(k_1, k_2; n)^m$ is dense if and only if one of the following holds:
 \begin{enumerate}
 \item $k_1 \leq j_2(m)$; or
 \item $0 \leq j_2(m) < k_1$ and $m(k_1-j(m))(k_2-k_1) < (k_2- j(m))^2$.
 \end{enumerate}
 \end{corollary}
 
 \begin{proof}
Set $j_2 = j_2(m)$.  By Lemma \ref{lem-easydense}, the $\PGL(n)$ action on $F(k_1, k_2; n)^m$ is dense if $n \geq m k_2$. Otherwise, by assumption, $0 \leq j_2< k_2$. If $k_1 \leq j_2$, then by $m$ applications of Lemma \ref{lem-reduce}, $$F(k_1, k_2; n)^m \Leftrightarrow \Gr(k_2-j_2, (m-1)(k_2-j_2))^m.$$  The action on the latter is dense by Theorem \ref{thm-popov}. If $0 \leq j_2 < k_1$, then by $m$ applications of Lemma \ref{lem-reduce}, $$F(k_1, k_2; n)^m \Leftrightarrow F(k_1 -j_2, k_2-j_2; (m-1)(k_2-j_2))^m.$$  
By Proposition  \ref{prop-m1},
 $$F(k_1 -j_2, k_2-j_2; (m-1)(k_2-j_2))^m \Leftrightarrow \Gr(k_1-j_2, k_2 - j_2)^m.$$ 
 The action on the latter is dense if and only if  $m (k_1-j_2)(k_2-k_1) < (k_2-j_2)^2$ by Theorem \ref{thm-popov}. This concludes the proof.
 \end{proof}
 
 Similarly, applying Theorem \ref{thm-mkr} when $r=3$ and $m=3$, we obtain the following.
 
 \begin{corollary}
 Let $n \geq 2k_3$. Then the $\PGL(n)$ action on $F(k_1, k_2, k_3; n)^3$ is dense if and only if one of the following holds:
 \begin{enumerate}
  \item $k_1 \leq j_3(3)$; or
 \item $0\leq  j_3(3) < k_1$ and $k_1 + k_2 \not= k_3 + j_3(3)$.
 \end{enumerate}
 \end{corollary}
 
 \begin{proof}
 For simplicity set $j_3(3) = j_3$. By assumption, $n \geq 2k_3$. If $n \geq 3k_3$, or equivalently, if $j_3 \geq k_3$, then the action on $F(k_1, k_2, k_3; n)^3$ is dense by Lemma 
\ref{lem-easydense}.  Otherwise, we can write $n = 2k_3 + j_3$ for some $0 \leq j_3 < k_3$. In this case, we apply Lemma \ref{lem-reduce} three times. If $j_3 < k_1$, then  $$F(k_1, k_2, k_3; n)^3 \Leftrightarrow F(k_1- j_3, k_2-j_3, k_3 - j_3; 2k_3-2j_3)^3.$$ By Proposition \ref{prop-m1}, 
$$F(k_1- j_3, k_2-j_3, k_3 - j_3; 2k_3-2j_3)^3 \Leftrightarrow F(k_1 - j_3, k_2-j_3; k_3-j_3)^3.$$ 
By Theorem \ref{thm-three2step}, the action on the latter is dense if and only if $k_1 - j_3 + k_2 -j_3 \not= k_3 -j_3$ or equivalently if $k_1 + k_2 \not= k_3 + j_3$. If $j_3 \geq k_2$, then applying Lemma \ref{lem-reduce} three times, $$F(k_1, k_2, k_3; n)^3 \Leftrightarrow \Gr(k_3-j_3, 2(k_3 - j_3))^3.$$ The action on the latter is always dense by Theorem \ref{thm-Grassmannian4}. If instead, $k_1 \leq j_3 < k_2$, then applying Lemma \ref{lem-reduce} three times, $$F(k_1, k_2, k_3; n)^3 \Leftrightarrow F(k_2-j_3, k_3-j_3; 2(k_3 - j_3))^3.$$ By Proposition \ref{prop-m1}, $$F(k_2-j_3, k_3-j_3; 2(k_3 - j_3))^3 \Leftrightarrow \Gr(k_2-j_3, k_3-j_3)^3.$$ The action on the latter is dense by Theorem \ref{thm-Grassmannian4}. We thus conclude that if $j_3 \geq k_1$, then the action on $F(k_1, k_2, k_3; n)^3$ is dense. 
 \end{proof}

\subsection{The case $k_{i+1} - k_i=1$}  Let $k_i = \ell + i$ for $1 \leq i \leq r$. We will denote this partial flag variety by $F(\ell+1, \dots, \ell+r; n)$ for simplicity.

\begin{theorem}\label{thm-1apart}
Let $k_i= \ell + i$ for $1 \leq i \leq r$. Then the action on $F(\ell+1, \dots, \ell+r; n)^3$ is dense if and only if one of the following holds
\begin{enumerate}
\item $r=1$; or
\item $3\ell + 3r \leq n$; or
\item $n= 3\ell + 3r-2$; or
\item $2\ell + 2r \leq n \leq 3\ell + 2r$ and either $r=2$ or  $3\ell + 5\geq n$; or
\item $n \leq 2\ell + r$ and $2n \leq 3\ell + 3$
\item $n \leq 2\ell + r$ and $2n = 3\ell + 5$
\item $2n-2\ell -2 \leq n \leq 3n-3\ell-r-3$ and either $r=2$ or $2n \geq 3\ell + 3r -2$.
\end{enumerate}
\end{theorem}
 
 \begin{proof}
\noindent {\bf Step 1.} The $\PGL(n)$ action on the product $F(\ell+1, \dots, \ell+r; n)^3$ is dense if $r=1$ by Theorem \ref{thm-Grassmannian4}. This is Case (1) of the theorem.
We may therefore assume that $r\geq2$.

\noindent {\bf Step 2.}  First, assume that  $n -\ell - 1 \geq \ell + r$, equivalently, $n \geq 2\ell + r + 1$.   If $n = 2\ell + t$ for $r+1 \leq t  \leq 2r-1$, then the action on $F(\ell+1, \dots, \ell+r; n)^3$  is not dense by Corollary \ref{cor-three2step}. Hence, we may assume that $n \geq 2\ell + 2r$. If $n \geq 3\ell + 3r$, then the action on $F(\ell+1, \dots, \ell+r; n)^3$ is dense by Lemma \ref{lem-easydense}. This is Case (2) of the theorem.

If instead  $2\ell + 2r \leq n < 3\ell + 3r$, then let $j = n - 2\ell - 2r$. 

{\bf Case 1: $j \leq \ell$.} If $j \leq \ell$, then by 3 applications of Lemma \ref{lem-reduce} followed by an application of Proposition \ref{prop-m1}, we have  
\begin{align*}
F(\ell+1, \dots, \ell+r; n)^3 &\Leftrightarrow F(\ell+1-j, \dots, \ell+r-j; 2(\ell+r-j))^3 \\ & \Leftrightarrow F(\ell+1-j, \dots, \ell+r-1-j; \ell+r-j)^3.
\end{align*}
By Lemma \ref{lem-duality},
$$F(\ell+1-j, \dots, \ell+r-1-j, \ell+r-j)^3 \Leftrightarrow F(1, \dots, r-1; \ell+r-j)^3.$$
By Theorem \ref{thm-1r}, this is dense if and only if $r=2$ or $r>2$ and $\ell + r - j \geq 3r-5$. The latter inequality is $3\ell + 5 \geq n$ and yields Case (4) of the theorem.

{\bf Case 2: $\ell+r > j > \ell$.} If $j > \ell$, then by 3 applications of Lemma \ref{lem-reduce} followed by an application of Lemma \ref{lem-duality}, we have 
\begin{align*}
F(\ell+1, \dots, \ell+r; n)^3 &\Leftrightarrow F(1, 2, \dots, \ell+r-j; 2(\ell+r-j))^3 \\ &\Leftrightarrow F(1, 2, \dots, \ell+r-1-j;  \ell+r-j)^3
\end{align*}
The latter is dense if and only if $\ell + r - j = 2$ by Theorem \ref{thm-1r}. This is Case (3) of the theorem.

\noindent {\bf Step 3.} If $n -\ell - 1 < \ell + r$, or equivalently, $n \leq 2\ell + r$, then by Lemma \ref{lem-duality} we have 
$$F(\ell+1, \dots, \ell+r; n)^3 \Leftrightarrow F(n-\ell-r, \dots, n-\ell -1; n)^3.$$ Setting $\ell' = n-\ell -r -1$,  this becomes $F(\ell' + 1, \dots , \ell' + r; n)^3$ and we can apply Step 2.  The remaining cases (5), (6), (7) follow by this duality. 
 \end{proof}
 
 When $m \geq 4$, we obtain the following.
 
 \begin{proposition}\label{prop-1apart}
 Assume that $m \geq 4$ and $n \geq (m-1) (\ell + r)$.  Set $j = n - (m-1)(\ell+r)$. Then $F(\ell+1, \dots, \ell+r; n)^m$ is dense if and only if one of the following holds:
 \begin{enumerate}
  \item $n \geq m(\ell + r)$; or
 \item $r=1$ and $m(\ell+1)(n-\ell-1) < n^2$; or
  \item $0 \leq j < \ell$ and either $r=2$ and $m \leq \ell-j+r+1$ or $r > 2$ and $\ell-j+r \geq mr-m-1$.
 \end{enumerate}
 
 \end{proposition}
 
 \begin{proof}
 If $n \geq m(\ell + r)$, then $F(\ell+1, \dots, \ell+r; n)^m$ is dense by Lemma \ref{lem-easydense}. We may therefore assume that $n < m (\ell + r)$.  If $r=1$, then $F(\ell+1, \dots, \ell+r; n)^m$ is dense if and only if $m (\ell + 1)(n-\ell -1) < n^2$ by Theorem \ref{thm-popov}. We may therefore assume $r>1$.  Set $j= n - (m-1) (\ell + r)$.
 By assumption $0 \leq j < \ell + r$.
 
 If $0 \leq j < \ell$, then by $m$ applications of Lemma \ref{lem-reduce} followed by an application of Proposition \ref{prop-m1} and Lemma \ref{lem-duality}, we have
 \begin{align*}
 F(\ell+1, \dots, \ell+r; n)^m &\Leftrightarrow F(\ell-j+1, \dots, \ell-j+r; (m-1)(\ell-j+r))^m \\ &\Leftrightarrow F(\ell-j+1, \dots, \ell - j + r-1; \ell -j + r)^m \\ & \Leftrightarrow F(1, \dots, r-1; \ell-j+r)^m
  \end{align*}
 By Theorem \ref{thm-1r}, the latter is dense if and only if $r=2$ and $m \leq \ell-j+r+1$ or $r> 2$ and $\ell-j+r \geq mr-m-1$.
 
 If $\ell \leq j < \ell+r$, then by $m$ applications of Lemma \ref{lem-reduce} followed by an application of Proposition \ref{prop-m1}, we have 
  \begin{align*}
 F(\ell+1, \dots, \ell+r; n)^m &\Leftrightarrow F(1, \dots, \ell-j+r; (m-1)(\ell-j+r))^m \\ &\Leftrightarrow F(1, 2, \dots, \ell-j+r-1; \ell -j + r)^m
  \end{align*}
  Since $m \geq 4$, the latter is never dense. This concludes the proof of the proposition.
 \end{proof}

 \section{Sparsity results}\label{sec-sparse}
In this section, we discuss further examples of the density and sparsity of products of partial flag varieties.

\begin{proposition} \label{prop-sparse}
Let $k_i < k_j$ be two indices such that $k_j = n-s$, where $s = k_i t + u$ for some $t \geq 1$ and $0 \leq u < k_i$. Then  the $\PGL(n)$ action on $F(k_1, \dots, k_r; n)^m$ is not dense if one of the following holds:
\begin{enumerate}
\item $u=0$ and $m \geq \max(t+2, 4)$.
\item $0 < u \leq  \frac{k_i}{2}$, $t\geq 3$  and $m \geq t+2$.
\item $\frac{k_i}{2} < u <k_i$, $t\geq 2$  and $m \geq t+3$.
\end{enumerate}
\end{proposition}

\begin{proof}
There is a $\PGL(n)$-equivariant surjective morphism $F(k_1, \dots, k_r; n)^m \to F(k_i, k_j; n)^m$. If the action on the latter is not dense, then the action on the former cannot be dense either. Hence, for the rest of the proof we will assume that $i=1, j=2$.  

Let $U_i \subset V_i$ for $1 \leq i \leq m$ be $m$ general partial flags of type $k_1, k_2$. 
Let $W$ be the span of $U_1$ and $U_2$. Observe that $W$ is a $(2k_1)$-dimensional linear space. 

\noindent {\bf Case 1: $u=0$.} Assume $u=0$ and  $m \geq \max(4, t+2)$.  In this case, let $S_3$ be the  span of $V_3$ and $U_i$ for $4 \leq i \leq t+2$. Let $S_4$ be the span of $V_4$, $U_i$ for $5 \leq i \leq t+2$ and $U_3$ if $t>1$. Let $U_i' = S_i \cap W$ for $i=3,4$. Observe that $U_3'$ and $U_4'$ are $k_1$-dimensional linear spaces in $W$.

 It is easy to see that we can get any two  general  $k_1$-dimensional linear spaces in $W$ in this way by choosing the flags $U_i \subset V_i$ appropriately. Namely, let $U_3'$ and $U_4'$ be general $k_1$-dimensional linear spaces in $W$.  If $t>1$, let $W'$ be a general $((t-2)k_1)$-dimensional subspace; otherwise let $W'=0$. Let $V_3$ be a general $k_2$-dimensional linear space intersecting the span of $W'$ and $U_4'$ in a $k_1$-dimensional linear space $U_3$. Similarly, let $V_4$ be a general $k_2$-dimensional linear space intersecting the span of $W'$ and $U_3'$ in a $k_1$-dimensional linear space $U_4$. Let $U_i$ for $i>4$ be general $k_1$-dimensional linear spaces in $W'$ and $V_i$ for $i > 4$ be general $k_2$-dimensional linear spaces containing $U_i$.  This reverses the construction. Since the action on $\Gr(k_1, 2k_1)^4$ is not dense, we conclude that  the action on $F(k_1, k_2; n)^m$ cannot be dense.

\noindent {\bf Case 2: $1 \leq u \leq \frac{k_1}{2}$.}  Assume $0< u \leq \frac{k_1}{2}$, $t \geq 3$ and and $m \geq t+2$. In this case, let $S_3$ be the  span of $V_3$ and $U_i$ for $4 \leq i \leq t+2$. Let $S_4$ be the span of $V_4$, $U_3$ and $U_i$ for $5 \leq i \leq t+2$. Finally, let $S_5$ be the span of $V_5$, $U_3, U_4$ and $U_i$ for $6 \leq i \leq t+2$. Let $U_i' = S_i \cap W$ for $3 \leq i \leq 5$. Observe that $\dim(U_i')= k_1-u$. We can get any three general $k_1$-dimensional linear spaces in $W$ by choosing the flags $U_i \subset V_i$ appropriately. Namely, for $3 \leq i \leq 5$, fix a general  $(n-s+(t-1)k_1)$-dimensional linear space $T_i$ intersecting $W$ in $U_i'$. Let $U_i$ for $3 \leq i \leq t$ be general linear spaces of dimension $k_1$ contained in $\cap_{i=3}^5 T_i$. For $3 \leq i \leq 5$, let $V_i$ be general linear spaces containing $U_i$ and contained in $T_i$. For $6 \leq i \leq t$. let $V_i$ be general linear spaces containing $U_i$. With these choices, we recover $U_3', U_4', U_5'$.

We have that 
$$\dim(\Gr(k_1, 2k_1)^2 \times \Gr(k_1 -u, 2k_1)^3) = 2k_1^2 + 3k_1^2 - 3u^2 \geq \frac{17}{4} k_1^2 > 4k_1^2.$$
Hence, the action on $F(k_1, k_2; n)^m$ cannot be dense.

\noindent {\bf Case 3: $\frac{k_1}{2} < u \leq k_1$.} Assume $\frac{k_1}{2} < u \leq k_1$, $t \geq 2$ and $m \geq  t+3$.  In this case, let $S_3$ be the  span of $V_3$ and $U_i$ for $4 \leq i \leq t+3$. Let $S_4$ be the span of $V_4$, $U_3$ and $U_i$ for $5 \leq i \leq t+3$. Finally, let $S_5$ be the span of $V_5$, $U_3, U_4$ and $U_i$ for $6 \leq i \leq t+3$. Let $U_i' = S_i \cap W$ for $3 \leq i \leq 5$. Observe that $\dim(U_i')= 2k_1-u$. As in the previous cases, we can get any three general $(2k_1-u)$-dimensional linear spaces in $W$ by choosing the flags $U_i \subset V_i$ appropriately. We have that 
$$\dim(\Gr(k_1, 2k_1)^2 \times \Gr(2k_1 -u, 2k_1)^3) = 2k_1^2 + 3u(2k_1-u) \geq \frac{17}{4} k_1^2 > 4k_1^2.$$
Hence, the action on $F(k_1, k_2; n)^m$ cannot be dense.

\end{proof}

\begin{corollary}
The $\PGL(n)$ action on  $F(k, n-2k; n)^m$ is dense if and only if $m \leq 3$.
\end{corollary}

\begin{proof}
By Theorem \ref{thm-three2step}, $F(k,n-2k; n)^m$ is dense if $m \leq 3$. By Proposition \ref{prop-sparse}, $F(k, n-2k; n)^m$ is sparse if $m \geq 4$.
\end{proof}

\begin{corollary}
Assume $m \geq 4$ and $2t <n$. If the $\PGL(n)$ action on $F(1, n-t ; n)^m$ is dense, then $m \leq t+1$. If $t \geq 3$ and $n> 2t$,  then the $\PGL(n)$ action on  $F(1, n-t; n)^{t+1}$ is dense if and only if $n \geq t(t+1)$.
\end{corollary}
 
 \begin{proof}
 By Proposition \ref{prop-sparse}, the $\PGL(n)$ action on  $F(1, n-t ; n)^m$ is not dense if $m \geq t+2$. Now suppose $m=t+1$. We have
 $$\dim(\PGL(n)) - \dim(F(1, n-t; n)^{t+1}) = n^2 -1 - (t+1)((n-t-1) + t (n-t)).$$
 When $n=2t+1$, this quantity is $-t^3+t^2+2t$, which is negative for $t \geq 3$.
When $n= t(t+1)-1$, this quantity is $(t+1)(2-t)$, which is negative for $t \geq 3$. We conclude that, under our assumptions,  the $\PGL(n)$ action on $F(1, n-t; n)^{t+1}$ cannot be dense unless $n \geq t(t+1)$. 

To prove the converse suppose that $n \geq t(t+1)$. By Lemma \ref{lem-duality}, it suffices to prove the density of the action on $F(t, n-1; n)^{t+1}$. By Lemma \ref{lem-basiclem}, it suffices to exhibit a set of $t+1$ partial flags $U_i \subset V_i$ for $1 \leq i \leq t+1$ with  $\dim(U_i)=t$ and $\dim(V_i) = n-1$ such that the stabilizer of the configuration has dimension $n^2 -1 - (t+1)(t(n-1-t) + n-1).$  Choose a basis $e_1, \dots, e_n$ for the underlying vector space. Since $n\geq t(t+1)$, for each $1 \leq i \leq t+1$, we can let the subspace $U_i$  be the coordinate subspace spanned by the basis elements $e_s$ with $(i-1)t<s \leq it$. We then let the $(n-1)$-dimensional flag element $V_i$ be defined by the linear relation $$\sum_{m=0}^{i-2}x_{mt+i-1} + \sum_{m=i}^{t} x_{mt + i}=0.$$ Matrices $M=(a_{u,v})$ stabilizing this configuration satisfy the following conditions:
\begin{enumerate}
\item $a_{u,v}=0$ if $u\leq t(t+1)$ and $v\not=u$
\item Let $0<m\leq t+1$. Then $a_{mt+1, mt+1} = a_{t+1, t+1}$. If $1 < s < m$, then $a_{mt+s, mt+s}= a_{s-1,s-1}$. If $m \leq s < t$, then $a_{mt+s, mt+s}= a_{s,s}$.
\item If $v > t(t+1)$, then for each $1 \leq i \leq t+1$, we have $\sum_{m=0}^{i-1} a_{mt+i-1,v} + \sum_{m=i}^t a_{mt+i,v} =0.$
\end{enumerate}
As each of the equations in item (3) involve distinct variables, they are independent. Hence, it is easy to see that this subgroup of $\PGL(n)$ has dimension $t + (n-t-1)(n-t(t+1))$. Observe that 
\begin{align*}
\dim(\PGL(n)) - (t+1) \dim(F(t, n-1; n)) &= n^2-1 - (t+1)(t(n-1-t) + n-1) \\ &= n^2 + t (t+1)^2 - nt(t+1)- n(t+1) + t \\ &= (n-t-1)(n-t(t+1)) +t  
\end{align*}
This concludes the proof. 
 \end{proof}

 \begin{proposition}\label{prop-obstruction}
Assume that there exist three indices $0<i_1< i_2 < i_3 \leq r$ such that $$2 k_{i_1} + k_{i_2} + k_{i_3}=2n.$$ Then the $\PGL(n)$ action on $F(k_1, \dots, k_r; n)^3$ is not dense.
\end{proposition}

\begin{proof}
There exists a $\PGL(n)$-equivariant  surjective morphism $$\phi: F(k_1, \dots, k_r; n)^3 \to F(k_{i_1}, k_{i_2}, k_{i_3}; n)^3.$$  If the $\PGL(n)$ action on  $F(k_{i_1}, k_{i_2}, k_{i_3}; n)^3$ is  not dense, then the action on  $F(k_1, \dots, k_r; n)^3$ cannot be dense either. Hence, it suffices to prove the proposition for $F(k_1, k_2, k_3; n)^3$ when $2k_1 + k_2 + k_3 = 2n$. 

For $1 \leq i \leq 3$, let $W_{i, k_i}$ denote the elements of three general  partial flags.  For $i=2,3$, let $U_i = W_{i, k_3} \cap W_{1,k_1}$.  Observe that $$\dim(U_i) = k_1+k_3 - n,$$ which is positive since $k_1 + k_3 > k_1 + k_2$ and $2k_1 + k_2 + k_3 = 2n$ by assumption. Let $$V_2 = \Span(W_{2, k_1}, W_{3,k_2}) \cap W_{1,k_1}\quad \mbox{and} \quad V_3= \Span(W_{3, k_1}, W_{2,k_2}) \cap W_{1,k_1}.$$ Observe that $$\dim(V_i) = 2k_1 + k_2 -n,$$ which is also positive since $n> k_3$ and $2k_1 + k_2 + k_3 = 2n$. 
We thus obtain a collection of 4 linear spaces $$(U_2, U_3, V_2, V_3) \in  \Gr(k_1 + k_3 -n, k_1)^2 \times \Gr(2k_1 + k_2 -n, k_1)^2.$$
It is easy to see that a general set of such 4-tuples arise via this construction. Namely, given general $U_2, U_3$ and $V_2, V_3$ in $W_{1,k_1}$, pick two general $k_3$-dimensional linear spaces $W_{2, k_3}$ containing $U_2$ and $W_{3, k_3}$ containing $U_3$. Then pick two $(k_1 + k_2)$-dimensional linear spaces $T_2$ containing $V_2$ and $T_3$ containing $V_3$. The intersection of $T_2, T_3$ and $W_{i,k_3}$ has dimension $k_2$, hence we can choose a general $k_1$-dimensional subspace $W_{i, k_1}$ in this intersection. Similarly, the intersection of $T_2$ (respectively, $T_3$) with $W_{3, k_3}$ (respectively, $W_{2,k_3}$) has dimension $k_1 + k_2+ k_3 - n>k_2$. Hence, we can choose a general $k_2$-dimensional linear space $W_{3, k_2}$ containing $W_{3, k_1}$ (respectively, $W_{2, k_2}$ containing $W_{2, k_1}$) in this intersection. It is clear that applying the construction to these partial flags yields back $U_2, U_3, V_2$ and $V_3$. We conclude that the action on $F(k_1, k_2, k_3;n)^3$ is sparse if  the action on $\Gr(k_1 + k_3 -n, k_1)^2 \times \Gr(2k_1 + k_2 -n, k_1)^2$ is sparse. Since $$2(k_1+k_3 - n) + 2(2k_1+k_2 -n)= 2k_1,$$ we conclude that the latter is sparse by Theorem \ref{thm-Grassmannian4}. This concludes the proof of the proposition.
\end{proof}

\begin{corollary}
Let $a < b$ be two positive integers such that $a+b$ is even. Then $F(a,b, n- \frac{a+b}{2}; n)^3$ is not dense.
\end{corollary}

\begin{proof}
By Lemma \ref{lem-duality}, $$F\left(a,b, n- \frac{a+b}{2}; n\right)^3 \Leftrightarrow F\left(\frac{a+b}{2}, n-b, n-a, ; n\right)^3.$$ The latter is not dense by Proposition \ref{prop-obstruction}.
\end{proof}

\bibliographystyle{plain}

\begin{thebibliography}{ABCH}

\bibitem[Bo91]{Borel}
A. Borel, Linear Algebraic Groups, Springer-Verlag, New York (1991).


\bibitem[CEY23]{CEY}
I. Coskun, D. Eken, and C. Yun, $\PGL$-orbits in tree varieties, preprint.

\bibitem[CHZ15]{CHZ}
I. Coskun, M. Hadian, and D. Zakharov, Dense $\PP$GL-orbits in products of Grassmannians, J. Algebra 429 no. 1 (2015), 75--102.



\bibitem[De14]{Devyatov}
R. Devyatov, Generically transitive actions on multiple flag varieties, Int. Math. Res. Not. (IMRN), {\bf 11} (2014), 2972--2989.

\bibitem[Ha92]{harris}
J. Harris, Algebraic Geometry: A First Course, Springer (1992).

\bibitem[MWZ99]{MWZ}
P. Magyar, J. Weyman, and A. Zelevinsky, Multiple flag varieties of finite type, Adv. Math., 141 no. 1 (1999), 97--118.





\bibitem[Po04]{Popov}
V. L. Popov, Generically multiple transitive algebraic group actions, Proceedings of the International
Colloquium on Algebraic Groups and Homogeneous Spaces, Mumbai (2004), 481--523.

\bibitem[Po07]{Popov2}
V. L. Popov, Tensor product decompositions and open orbits in multiple flag varieties, J. Algebra 313
no. 1 (2007), 392--416.


\bibitem[Sm20]{Smirnov}
E. Yu. Smirnov, Multiple flag varieties, Journal of Mathematical Sciences, {\bf 248} no. 3 (2020), 338--373.

\bibitem[SW98]{SW}
G. W. Schwarz, and D. L. Wehlau, Invariants of four subspaces, Ann. Inst. Fourier, 48 no.3 (1998), 667--697.

\end{thebibliography}

\end{document}